\documentclass[11pt]{amsart}

\usepackage{amssymb, amsthm, amsmath}
\usepackage{amsfonts}
\usepackage{graphicx, comment}
\usepackage[small]{caption}
\usepackage{subcaption}
\usepackage{epsfig}
\usepackage{tikz, float}
\usepackage [english]{babel}
\usepackage[utf8]{inputenc}
\usepackage [autostyle, english = american]{csquotes}
\MakeOuterQuote{"}

\usepackage{complexity}

\newcommand{\later}[1]{}
\newcommand{\old}[1]{}

\usepackage{fullpage}
\usepackage{framed}

\usepackage{enumerate}
\usepackage{url}
\usepackage[breaklinks]{hyperref}
\hypersetup{
	colorlinks = true, 
	urlcolor = cyan, 
	linkcolor = teal, 
	citecolor = cyan 
}

\newtheorem{theorem}{Theorem}[section]
\newtheorem{lemma}[theorem]{Lemma}

\newtheorem{observation}[theorem]{Observation}
\newtheorem{definition}[theorem]{Definition}

\newtheorem{corollary}[theorem]{Corollary}

\newcommand{\ie}{{i.e.}}

\newcommand{\eps}{\varepsilon}

\newcommand{\Z}{\mathbb{Z}} 
\newcommand{\RR}{\mathbb{R}} 

\def\A{\mathcal A}

\def\C{\mathcal C}
\def\D{\mathcal D}

\usepackage{tikz}
\tikzstyle{vtx} = [circle, fill, inner sep=0.7]

\title{Rado's covering problem for cubes and balls: a survey}

\thanks{The first named author would like to acknowledge the financial support
  of the Istituto Nazionale di Alta Matematica ``F. Severi".
Work by the second named author was supported in part by
ERC Advanced Grant `GeoScape' No. 882971 and by the Erd\H os Center.}

\author{Gian Maria Dall'Ara}
\address{Gian Maria Dall'Ara \newline Istituto Nazionale di Alta Matematica ``F. Severi"
  and Scuola Normale Superiore, Pisa, Italy}
\email{dallara@altamatematica.it}

\author{Adrian Dumitrescu}
\address{Adrian Dumitrescu \newline
Algoresearch L.L.C., Milwaukee, WI, USA, and 
Research Institute of the University of Bucharest, Romania, and 
Alfr\'ed R\'enyi Institute of Mathematics, Budapest, Hungary}
\email{ad.dumitrescu@algoresearch.org}

\date{\today}

\begin{document}
	
\begin{abstract} 
	What is the largest constant $c\in [0,1]$ with the property that every finite collection
        $\mathcal{C}$ of axis-parallel 
	squares in the plane admits a disjoint sub-collection $\mathcal{S}$ occupying at least a
        fraction $c$ of the area covered by $\C$? This problem was first raised by T.~Radó in 1928,
        who was motivated by a classical covering lemma in real analysis due to Vitali. R.~Rado later
        generalized the problem from axis-parallel squares in the plane to homothetic copies of any
        given convex body $K$ in $\RR^d$, where now we are looking for an optimal constant $F(K)$.  

        In this paper we survey the literature on this problem, with a focus on the case where $K$ is a cube or a ball in the high-dimensional regime $d\rightarrow \infty$.
        The estimates that we currently have for cubes are much more precise
        than those for balls: namely if $Q^d$ is a $d$-dimensional cube, then
        \[ (e^{-1}+o(1))\frac{2^{-d}}{d \log{d}} \leq F(Q^d)\leq 2^{-d},  \]
        while denoting $B^d$ a $d$-dimensional Euclidean ball, then
        \[ 2.910^{-d} \leq F(B^d) \leq 2.447^{-d}. \]
        Except for the upper bound for cubes, which is classical, the other estimates are due to the authors. 
	
		
\end{abstract}

\maketitle

\section{A short introduction} 

Every finite collection $\mathcal{C}$ of compact intervals on the real line admits a disjoint sub-collection
$\mathcal{S}$ occupying at least half of the total length covered by $\mathcal{C}$. This innocent-looking
statement solves the simplest instance of the Rado's covering problem in the title.\footnote{As it turns out,
  this is the only instance of the problem for which an exact solution is known.} The reader may find a proof
of this statement in Section \ref{sec:origin} below, or try to prove it on its own.

This paper surveys the existing literature on this problem, starting from its origins in early twentieth
century real analysis up to recent developments. This occupies Sections \ref{sec:origin} to
\ref{sec:surprise}. Then, in Section~\ref{sec:balls},  we present a new result,
highlighting a connection with sphere packing. The main focus is in the asymptotic regime where
the dimension of the ambient Euclidean space tends to $\infty$. Remarks on algorithmic aspects and a list of
open problems are included in Section \ref{sec:remarks}.

\section{The origin story}\label{sec:origin}

In 1908, G.~Vitali wrote an innovative paper~\cite{Vitali1908}, where he proved a version of
the fundamental theorem of calculus in the (then) brand new Lebesgue theory of integration.
For this purpose, he introduced a new kind of tool to the analyst's kit, now known as a
\emph{covering lemma}.

Vitali's original covering lemma, in its finitary version, is the following theorem
(cf.~\cite[p.~233]{Vitali1908}), where we denote by $|E|$ the usual $d$-dimensional volume of a (measurable)
set $E\subseteq \RR^d$, and call \emph{axis-parallel cube} any set of the form
\[  Q(x,r)=[x_1-r, x_1+r]\times \cdots \times [x_d-r, x_d+r],\qquad x=(x_1,\ldots, x_d)\in \RR^d, \, r>0. \] 

\begin{theorem}[Original Vitali covering lemma]\label{thm:vitali}
	Any finite collection $\mathcal{C}$ of axis-parallel cubes in $\RR^d$ admits a disjoint
        sub-collection $\mathcal{S}$ such that  
\begin{equation} \label{eq:vitali} 
	|\cup_{Q\in \mathcal{S}}Q|\geq 3^{-d}|\cup_{Q\in \mathcal{C}}Q|. 
\end{equation}
\end{theorem}

E.g., the theorem states that, given any finite collection of axis-parallel cubes in $3$-space, one 
can always select some of them, so that no two cubes intersect and
  the selected cubes occupy at least $1/27$ of the volume of the original collection.

Vitali's constructive proof is a textbook example of a \emph{greedy algorithm} (of course, this terminology
did not exist in the first decade of the 20th century). Let us recall how it works.  

\begin{proof}[Proof of Theorem \ref{thm:vitali}]
  The selection of the desired sub-collection is achieved by the following procedure:
  \begin{enumerate} 
\item choose a largest cube in the collection [with ties broken arbitrarily],
\item remove from the collection all cubes intersecting the selected one, 
\item repeat until the collection is empty. 
  \end{enumerate}
It is evident that this algorithm produces a disjoint sub-collection. Moreover, any discarded cube $Q'$ must
intersect a \emph{larger selected cube} $Q$. In this case, $Q'$ is contained in $3Q$, that is, the cube with
same (bary)center and three times the side-length, see Fig.~\ref{fig:squares}\,(left).

\begin{figure}[htbp]
 \centering
 \includegraphics[scale=0.7]{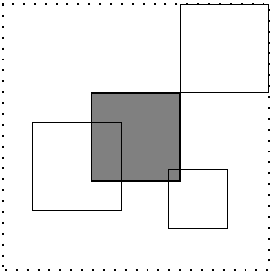}
\hspace{25mm}
 \includegraphics[scale=1.1]{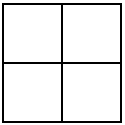}
 \caption{Left: a largest square (shaded) and three others that intersect it.
Right: four congruent squares sharing a common point.}
 \label{fig:squares}
\end{figure}

Thus, letting $\mathcal{S}$ be the sub-collection of selected cubes,
\[
|\cup_{Q'\in \mathcal{C}}Q'|\leq |\cup_{Q\in \mathcal{S}}3Q|\leq \sum_{Q\in \mathcal{S}}|3Q|.
\]
Now observe that if a set in $\RR^d$ is dilated by a factor $\lambda>0$, its volume gets multiplied by a factor
$\lambda^d$, so (taking into account the disjointness of $\mathcal{S}$) the rightmost term in the above chain
of inequalities is $3^d|\cup_{Q\in \mathcal{S}}Q|$. 
\end{proof} 

For the application that Vitali had in mind, the constant $3^{-d}$ plays no significant role,
assuming that the dimension $d$ is fixed: all that matters is that the fraction of volume
of the selected sub-collection is bounded from below by a positive constant
depending only on the dimension $d$. Yet, many mathematicians, when confronted
with the statement of Theorem~\ref{thm:vitali}, have been instinctively led to the question:
\emph{is the constant $3^{-d}$ optimal and, if not, what is the value, or at least the order of magnitude,
  of the best possible constant?} 

\smallskip
The first recorded instance of the above question seems to be a letter of T.~Rad\'o to W.~Sierpinski,
an excerpt of which appeared in print in 1928~\cite{Radob1928}. In his letter, the Hungarian mathematician
proves a sharp one-dimensional covering lemma with the constant improved from $\frac{1}{3}$ to $\frac{1}{2}$,
and remarks that his proof is "worth being explicitly pointed out, since it is based on an entirely
different principle from the usual one"\footnote{The original is in French.} (that is, Vitali's greedy argument).
Let us reproduce the result and its proof.  

\begin{theorem}[Sharp one-dimensional covering lemma]\label{thm:tibor_rado}
  Any finite collection $\mathcal{C}$ of compact intervals in $\RR$ admits a disjoint sub-collection $\mathcal{S}$
  such that 
\begin{equation}\label{eq:tibor_rado}
	|\cup_{I\in \mathcal{S}}I|\geq \frac{1}{2}|\cup_{I\in \mathcal{C}}I|. 
\end{equation}
The constant $\frac{1}{2}$ is sharp.
	\end{theorem}

\begin{proof}
  Let $\mathcal{C}=\{I_1,\ldots, I_N\}$. We may assume without loss of generality that
no interval is contained in the union of the others (by removing the covered intervals). 
 Equivalently, each interval $I_j$ contains a point $\xi_j$ that lies in no other interval of $\mathcal{C}$.
 We may also assume that the labels of the intervals are chosen so that $\xi_1<\xi_2<\ldots<\xi_N$.
 If two indices $i<j$ are not adjacent (that is, if $j>i+1$), then the corresponding intervals
 $I_i$ and $I_j$ must be disjoint, otherwise their union would cover $\xi_k$ for each $k$ between
 $i$ and $j$, which is not the case by assumption. In particular, the two sub-collections
\[ \mathcal{C}_0=\{I_j\colon\, j \, \textrm{even}\}, \qquad \mathcal{C}_1=\{I_j\colon\, j \, \textrm{odd}\}
\] are disjoint. One of the two must cover at least a half of (the measure of) $\cup_jI_j$.
This proves~\eqref{eq:tibor_rado}.
For the sharpness, consider $\mathcal{C}=\{[0,1], [1,2]\}$; or a larger family of this kind.
\end{proof}

The above elementary proof indeed rests on a fundamentally different principle. As highlighted in the next section,
Theorem~\ref{thm:vitali} depends solely on volume scaling and the triangle inequality,
while Rad\'o's argument exploits the order (or topological) structure of the real line.  

Moreover, the proof of Theorem~\ref{thm:tibor_rado} has a "global" quality: while Vitali's procedure makes at
each step a locally optimal (a.k.a.~"greedy") choice, Rad\'o's argument prepares two candidate sub-collections,
at least one of which attains the desired bound. Deciding which of the two does the job requires computing
their total measures. 

\begin{figure}[htbp]
 \centering
 \includegraphics[scale=0.8]{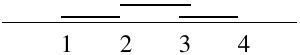}
 \caption{The family of intervals $[0,1], [2,3+\eps], [3,4]$ on the line, where Rad\'o's strategy
   outperforms Vitali's.}
 \label{fig:line}
\end{figure}

In his letter to Sierpinski, Rad\'o also raises the problem of proving an analogous statement in two dimensions,
that is, determining the best possible constant $c$ for which any finite collection $\mathcal{C}$ of
axis-parallel squares in $\RR^2$ admits a disjoint sub-collection $\mathcal{S}$ such that  
\begin{equation}\label{eq:tibor_rado_conjecture}
	|\cup_{Q\in \mathcal{S}}Q|\geq c|\cup_{Q\in \mathcal{C}}Q|. 
\end{equation}
He also claims that the above inequality is "probably" true with $c=\frac{1}{4}$. Notice that the constant
cannot be larger than $\frac{1}{4}$, as the collection of squares in Fig.~\ref{fig:squares}\,(right) shows.

The role of the topology of $\RR$ in the proof of Theorem~\ref{thm:tibor_rado} led Rad\'o to state that "if the
[two-dimensional] problem seems to deserve interest, it is because of the important topological facts on which
its solution appears to depend." Before describing progress on Rad\'o's problem, let us place it in a more
general context.

\section{The Vitali covering lemma for a general convex body}\label{sec:general_bodies}

Inspecting the proof of Theorem~\ref{thm:vitali}, one is led to a few considerations. First of all,
axis-parallel cubes are the (closed) balls with respect to the $\ell^\infty$ norm in $\RR^d$:
\[ 
Q(x,r)=\{y\in \RR^d\colon\, \lVert y-x\rVert_{\ell^\infty}\leq r\}, 
\]
where \[
 \lVert x\rVert_{\ell^\infty}=\max_{j=1,\ldots, d}|x_j|. 
\]
The crucial inclusion $Q'\subseteq 3Q$ in the proof of the covering lemma is an instance of the following
general fact: \emph{if two balls $B(x,r)$ and $B(y,s)$ in a metric space intersect, and $s\leq r$, then
  $B(y,s)\subseteq B(x,3r)$}. This is of course an immediate corollary of the triangle inequality. It leads to
a metric space statement, which is one of the results that is nowadays customarily referred to as the "Vitali
covering lemma" (see, e.g.,~\cite[Lemma 7.3]{Rudin1987}).  

\begin{lemma}\label{lem:vitali_metric}
	Any finite collection $\mathcal{C}$ of closed balls in a metric space admits a disjoint sub-collection
        $\mathcal{S}$ such that $\cup_{Q\in \mathcal{C}}B\subseteq \cup_{Q\in \mathcal{S}}3B$. 
\end{lemma}

Here $3B$ is the ball with same center as $B$ and three times the radius.\footnote{In a completely general
  metric space, the center and radius of a ball, and hence $3B$, are not uniquely defined: the statement holds
  for \emph{any} such choice.}  

\smallskip
Let now $\lVert \cdot \rVert$ be any norm on $\RR^d$, or equivalently let \[K=\{x\in \RR^d\colon\, \lVert
x\rVert\leq 1\}\] be its unit ball, which is an arbitrary \emph{symmetric convex body} (i.e., a compact convex
set with non-empty interior, which is centrally symmetric with respect to the origin). Then the balls in the
metric space $(\RR^d, \lVert \cdot\rVert)$ are the homothetic copies $x+rK$ of $K$ ($x\in \RR^d$, $r>0$).  

Combining Lemma~\ref{lem:vitali_metric} with the volume scaling property $|x+rK|=r^d|K|$ (in other words,
repeating the proof of Theorem~\ref{thm:vitali}), we obtain the following generalization of the original
Vitali covering lemma.  

\begin{theorem}[Vitali covering lemma for symmetric convex bodies]\label{thm:vitali_norm}
  Let $K\subseteq \RR^d$ be a symmetric convex body. Any finite collection $\mathcal{C}$
  of homothetic copies of $K$ admits a disjoint sub-collection $\mathcal{S}$ such that  
	\begin{equation}\label{eq:vitali_norm}
		|\cup_{K'\in \mathcal{S}}K'|\geq 3^{-d}|\cup_{K'\in \mathcal{C}}K'|. 
	\end{equation}
\end{theorem}

As in the case of axis-parallel cubes, the problem of determining the best possible constant in Theorem
~\ref{thm:vitali_norm} presents itself. Indeed, this general "covering problem" has been introduced and
systematically investigated by R.~Rado in a 1949 paper~\cite{Rado1949I},
and its sequels~\cite{Rado1951II, Rado1968III}. Rado does not mention Vitali's work or Rad\'o's letter
(he attributes the greedy argument in the case of cubes to J.~C.~Burkill),
but considerations of historical character apart, he recognized the appropriate general framework
for these matters. The following definition comes from~\cite{Rado1949I}.  

\begin{definition}\label{def:F}
If $K$ is a symmetric convex body\footnote{In fact, Rado defined and studied the quantity $F(K)$ for $K$ not
  necessarily symmetric, but we won't consider this more general setting here.} in $\RR^d$,
let $F(K)$ denote the largest constant $c\geq 0$ for which the following statement holds true: any
finite collection $\mathcal{C}$ of homothetic copies of $K$ admits a disjoint sub-collection $\mathcal{S}$
such that  
\[
	|\cup_{K'\in \mathcal{S}}K'|\geq c|\cup_{K'\in \mathcal{C}}K'|. 
\]
\end{definition}

As one may easily check, the quantity $F(K)$ is invariant under the action of the general linear group $GL_d(\RR)$,
that is, $F(T(K))=F(K)$ for all $T\in GL_d(\RR)$ (notice that $T(K)$ is also a symmetric convex body).

\smallskip
By Theorem~\ref{thm:vitali_norm}, we have
\begin{equation} \label{eq:greedy_lower_bound}
F(K)\geq 3^{-d}
\end{equation}
for every symmetric convex body $K$.  
Rado showed that a minuscule improvement over this lower bound is possible, uniformly over all
$d$-dimensional symmetric convex bodies (Theorem 8 of~\cite{Rado1949I}). This may be achieved by
slightly modifying the greedy algorithm in the proof of Theorem~\ref{thm:vitali}; or, more
precisely, its adaptation to an arbitrary symmetric convex body $K$. The idea is roughly as follows.
At the current step of the algorithm, we have a largest homothetic copy $K'$ among the ones left
in the collection and a sub-collection $\mathcal{C}_{K'}$ of homothetic copies $K''$ intersecting~$K'$.
If $|\cup_{K''\in \mathcal{C}_{K'}}K''|$ is almost as large as $|3K'|$, then we may find two
disjoint sets in $\mathcal{C}_{K'}$ that are almost as large as $K'$, and choosing them is more
convenient than choosing $K'$. Otherwise, simply choose $K'$. Quantifying the "almost as large"
appropriately, Rado proved that
\begin{equation}\label{eq:minuscule_rado} 
F(K) \geq 3^{-d}(1+\epsilon_d), \qquad \text{with  } \epsilon_d=7^{-d}(d+2)^{-d^2-d} 
\end{equation}
for every symmetric convex body $K$. \smallskip 

On the other hand, an appropriate generalization of the construction in Fig.~\ref{fig:squares}\,(right)
yields the upper bound
\begin{equation}\label{eq:easy_upper_bound} 
F(K)\leq 2^{-d}.
\end{equation}
Indeed, consider the collection $\mathcal{C}$ of \emph{all translates of $K$ containing, say, the origin}.
The union of this collection is $2K$, and so its volume is $2^d|K|$.\footnote{This fact is
most easily seen observing that $\mathcal{C}$ is the collection of all balls of unit radius
(with respect to the norm corresponding to $K$) and center in $K$.}
Since a non-trivial disjoint sub-collection $\mathcal{S}$ of $\mathcal{C}$ must consist
of a single translate of $K$, the volume fraction occupied by $\mathcal{S}$ is equal to~$2^{-d}$.
This proves~\eqref{eq:easy_upper_bound}.\footnote{The reader may notice that $\mathcal{C}$ is
an infinite collection of balls, but this limitation is overcome by considering finite sub-collections
occupying any prescribed fraction of the volume of $\mathcal{C}$.} 

\smallskip
Thus, we have the following \emph{generalized Rad\'o's problem}, which is also referred to as
\emph{Rado's  covering problem} (e.g., by Wikipedia~\cite{Wikipedia}):
determine, or more realistically estimate, the
value of the constants $F(K)$. We remark in passing that we decided to stick to the established
"covering" terminology for this problem, even though one may argue that the questions under
consideration have the flavor of "packing problems" (see, e.g., Section D of~\cite{Croft_Falconer_Guy1991}),
as remarked by R.~Rado himself (see p.~233 of~\cite{Rado1949I}).
This point of view is especially pertinent in connection with our forthcoming
discussion of the case of Euclidean balls in Section \ref{sec:balls}.

In this paper, our main focus will be on the asymptotic behavior of these constants in
the high dimensional regime,
where the lower bound~\eqref{eq:greedy_lower_bound} (or~\eqref{eq:minuscule_rado})
and the upper bound~\eqref{eq:easy_upper_bound} are exponentially separated. 

\smallskip
Two cases that seem to deserve attention are \[K=Q^d:=[-1,1]^d,\] the $d$-dimensional cube,
and \[K=B^d:=\{x\in \RR^d\colon\, \sum_{j=1}^dx_j^2\leq 1\},\] the $d$-dimensional Euclidean ball.
As subsequently shown, the constants $F(Q^d)$ and $F(B^d)$ behave very differently in the limit $d\rightarrow \infty$.
In the case of cubes, the exponential rate of decay of the constants is now understood, while for balls this
seems to be a difficult and still widely open problem. In the next section we  start with the case of cubes.

\smallskip
From an opposite direction, it is worth making the following observation showing that
the condition of ``homothetic'' in Theorem~\ref{thm:vitali_norm} cannot be dropped.
Its verification is left to the reader.

\begin{observation} \label{obs:eps} 
Fix $d \geq 2$. For every $\eps>0$, there exists a convex body $K\subseteq \RR^d$ and a finite
collection $\C$ of similar copies of $K$ such that for any disjoint sub-collection $\mathcal{S}$ we have
\[ |\cup_{K \in \mathcal{S}} K | \leq \eps |\cup_{K \in \mathcal{C}}K|.  \]
Moreover, this holds with collections $\C$ of congruent bodies.
\end{observation}

\section{Progress for axis-parallel cubes}

The first progress made on these matters (after T.~Rad\'o's letter) is contained
in the short 1940 paper~\cite{Sokolin1940} by the Soviet mathematician A.~S.~Sokolin,
who recalls Rad\'o's conjecture (sic) and generalizes it to an arbitrary number of dimensions,
namely he states the conjecture that
\begin{equation}\label{eq:F-false} 
F(Q^d)=2^{-d}.
\end{equation}
Recall that $Q^d=[-1,1]^d$. He then gives a proof of a special case, where one limits
consideration to collections of \emph{congruent axis-parallel cubes}. A generalization of the same special
case, with the same proof, also appeared later in the already cited paper~\cite{Rado1949I} of R.~Rado, who
gives credit for this specific contribution to A.~S.~Besicovitch (see p.~233 and p.~254 of~\cite{Rado1949I}).

\smallskip
Before proceeding, it is convenient to introduce a variant of Definition \ref{def:F}, also coming from
R.~Rado's paper~\cite{Rado1949I}. 
 
\begin{definition}\label{def:f}
If $K$ is a symmetric convex body in $\RR^d$, let $f(K)$ denote the largest
constant $c\geq 0$ for which the following statement holds true: any finite collection $\mathcal{C}$
of congruent homothetic copies of $K$ admits a disjoint sub-collection $\mathcal{S}$ such that  
	\[
	|\cup_{K'\in \mathcal{S}}K'|\geq c|\cup_{K'\in \mathcal{C}}K'|. 
	\]
\end{definition}

In other words, $f(K)$ is defined considering only collections of the form \[\mathcal{C}=\{x_1+rK, \ldots,
x_N+rK\},\] where $x_1,\ldots, x_N\in \RR^d$ and $r>0$. By scale-invariance considerations, one may take $r=1$,
that is, one may consider only translates of $K$. Exactly as $F(K)$,
the quantity $f(K)$ is invariant under $GL_d(\RR)$. Clearly,
\begin{equation}\label{eq:fandF}
  f(K)\geq F(K).
  \end{equation} 
Rado (Theorem 8 of~\cite{Rado1949I}) proved a lower bound on $f(K)$ that is slightly better than
what one gets combining~\eqref{eq:fandF} and~\eqref{eq:minuscule_rado}. Again, one modifies the
Vitali greedy algorithm, but in a somewhat cleaner way (compared to the case of $F(K)$): at each
step one chooses a "boundary" set from the given collection, namely one that admits a supporting
hyperplane leaving the whole collection on one side; this is a variant of the "sweep line/hyperplane"
argument in computational geometry. It still yields only a modest improvement
over the Vitali bound:
\begin{equation}\label{eq:sweep} 
  f(K)\geq (3^d-2^{d-1})^{-1}.
\end{equation}

Moreover, \begin{equation}\label{eq:easy_upper_bound_2}
f(K)\leq 2^{-d}, 
\end{equation}
because the example in Section \ref{sec:general_bodies} proving~\eqref{eq:easy_upper_bound} consists of
translates of $K$. The partial progress on Rad\'o's problem alluded to above closes, in the case of cubes,
the exponential gap between~\eqref{eq:sweep} and~\eqref{eq:easy_upper_bound_2}, by determining the exact value
of the constant. 

\begin{theorem}[Sharp covering lemma for congruent axis-parallel cubes]\label{thm:sharp_covering_congruent_cubes}
  $f(Q^d)=2^{-d}$. 
\end{theorem}

By the invariance under the general linear group, the same statement is then true for any $d$-dimensional
parallelepiped.

\smallskip 
Sokolin's and Rado--Besicovitch's proof of Theorem~\ref{thm:sharp_covering_congruent_cubes}
(from~\cite{Sokolin1940} and~\cite{Rado1949I}) is a rapid application of the so-called \emph{Blichfeldt's  principle},
a basic result in the \emph{geometry of numbers}. Let us recall its statement.   

\begin{lemma}[Blichfeldt's principle~\cite{Blichfeldt14}]\label{lem:blichfeldt}
Let $\Lambda\subseteq \RR^d$ be a lattice of covolume $V$. If $E\subseteq \RR^d$ is a measurable set of
finite measure, then there exists a translate of $\Lambda$ that intersects $E$ in at least $|E|V^{-1}$ points.  
\end{lemma}
Recall that a lattice is any set of the form $\Lambda=T(\Z^d)$ for some $T\in GL_d(\RR)$. The covolume of
$\Lambda$ is the volume of the torus $\RR^d/\Lambda$ (with respect to the flat metric), or equivalently the
quantity $|\det(T)|$. The proof of Lemma~\ref{lem:blichfeldt} is very simple: the average number of points of
$E$ in a random translation of $\Lambda$ is equal to $|E|V^{-1}$. We refer to Blichfeldt's original paper for
more details.  

\begin{proof}[Sokolin's/Besicovitch--Rado's proof of Theorem~\ref{thm:sharp_covering_congruent_cubes}]
  Let  $\mathcal{C}$ be a finite collection of axis-parallel cubes, which we assume without loss of generality to
  be of unit side length (for short, "unit cubes"). Let $E$ be its union. Fix $\eps>0$. By Blichfeldt's
  principle, some translate of the lattice $(2+\eps)\Z^d$ contains at least \[(2+\eps)^{-d}|E|\]
  points of $E$. 

Now notice that if two cubes of $\mathcal{C}$ intersect two distinct elements of a translation of
$(2+\eps)\Z^d$, then they are necessarily disjoint (this is why we gave ourselves an $\eps$ of
room). Thus, $\mathcal{C}$ has a disjoint sub-collection consisting of at least $(2+\eps)^{-d}|E|$ elements.
This proves that $f(Q^d)\geq (2+\eps)^{-d}$. Letting $\eps$ tend to zero, we obtain the thesis.  
\end{proof}

\label{page:voronoi}The proof above works because $Q^d$ is the Voronoi cell (centered at the origin)
of the lattice $2\Z^d$. This allows to generalize the statement to Voronoi cells of any lattice. See
p.~256 of~\cite{Rado1949I} for this point. In fact, Rado uses the same idea to prove non-sharp lower
bounds for $f(K)$ for even more general bodies, e.g., \[f(B^2)\geq \frac{\pi}{8\sqrt{3}}.\] See
Theorem~10 of~\cite{Rado1949I}. 
  Note that the use of Blichfeldt's principle can be made
constructive and deterministic. That is, a suitable translation in Lemma~\ref{lem:blichfeldt}
can be efficiently computed, provided that the set $E\subseteq \RR^d$ has a ``nice'' description; 
see for instance~\cite[Thm.~1]{Bereg_Dumitrescu_Jiang2010a}.

\smallskip 
Two other proofs of Theorem~\ref{thm:sharp_covering_congruent_cubes} are known,
due to G.~Nordlander~\cite{Nordlander1958} and V.~A.~Zalgaller~\cite{Zalgaller1960}.
Since both are quite ingenious, we won't resist the temptation to describe them.  

\begin{proof}[Sketch of Nordlander's proof]
  We argue by induction on the dimension, with the base case ($d=1$) being Theorem~\ref{thm:tibor_rado}.
  Assume that $f(Q^{d-1})=2^{-d+1}$, and let $\mathcal{C}$ be a collection of axis-parallel unit cubes in $\RR^d$,
  with union $E$. Let $E_t=E\cap \{x_d=t\}$ be the slice of $E$ of height $t$. By the inductive assumption,
  one may select a disjoint sub-collection $\mathcal{S}(t)$ of $\mathcal{C}$ with the property
that \begin{equation}\label{eq:nordlander} 
  |\cup_{Q\in \mathcal{S}(t)}Q|_{d-1} \geq 2^{-d+1}|E_t|_{d-1},
\end{equation}
where $|\cdot|_m$ is $m$-dimensional Lebesgue measure. Now one observes that the collections $\mathcal{S}(t)$
corresponding to $t\in t_0+2\Z$ are disjoint, for all but finitely many values of $t_0\in [0,2)$. Let
  $\mathcal{S}'(t_0):=\cup_{t\in t_0+2\Z}\mathcal{S}(t)$. Now pick $t_0\in [0,2)$ uniformly at random.
    From~\eqref{eq:nordlander}, it is not difficult to see (by integration) that the average fraction
    of the volume of $E$ occupied by the cubes in $\mathcal{S}'(t_0)$ is at least $2^{-d}$. Thus,
    for an appropriate value of $t_0$, $\mathcal{S}'(t_0)$ is the desired sub-collection. 
\end{proof}

\begin{proof}[Sketch of Zalgaller's proof] We sketch the two-dimensional case, as the extension to general
  dimension presents no additional difficulty.  
	
One starts by observing that the conclusion is pretty easy to achieve for collections of unit squares in "grid
position", that is, of the form $[a,a+1]\times [b,b+1]$ with $a,b\in \Z$. In fact, any such collection can be
partitioned into $4$ disjoint sub-collections, e.g., based on the parity of $a$ and $b$. Clearly, one of these
sub-collections must occupy at least $\frac{1}{4}$ of the total area. 

The key claim is that any finite collection $\mathcal{C}$ of axis-parallel unit squares can be moved
  into grid position, without decreasing its total area and without dis-joining any pair of intersecting squares.
  In other words, for each square $Q\in \mathcal{C}$ there is a translation vector $v_Q\in \RR^2$ so
that the following hold: (1) if $Q$ intersects $Q'$, then $Q+v_Q$ intersects $Q'+v_{Q'}$; (2) the total area
of $\mathcal{C}'=\{Q+v_Q\colon\, Q\in \mathcal{C}\}$ is not smaller than that of $\mathcal{C}$; (3)
$\mathcal{C}'$ is in grid position. Using this process, the conclusion follows quite easily: translating back
the disjoint sub-collection of $\mathcal{C}'$ obtained in the previous paragraph, one obtains
the desired sub-collection of $\mathcal{C}$.  

The proof of the claim proceeds "one coordinate at a time", translating first horizontally and then
vertically. Fix a square $Q_0\in \mathcal{C}$, and say that a square $Q\in \mathcal{C}$ is \emph{horizontally
  aligned} if its vertical sides are at an integer distance from those of $Q_0$. Now translate horizontally,
as a rigid body, the horizontally aligned squares: we obtain new collections
\[
\mathcal{C}(h):=\{Q+(h,0)\colon\, Q\quad \text{horizontally aligned}\}\cup \{Q\colon\,
Q\quad \text{not horizontally aligned}\}.
\]

\begin{figure}[htbp]
 \centering
 \includegraphics[scale=0.4]{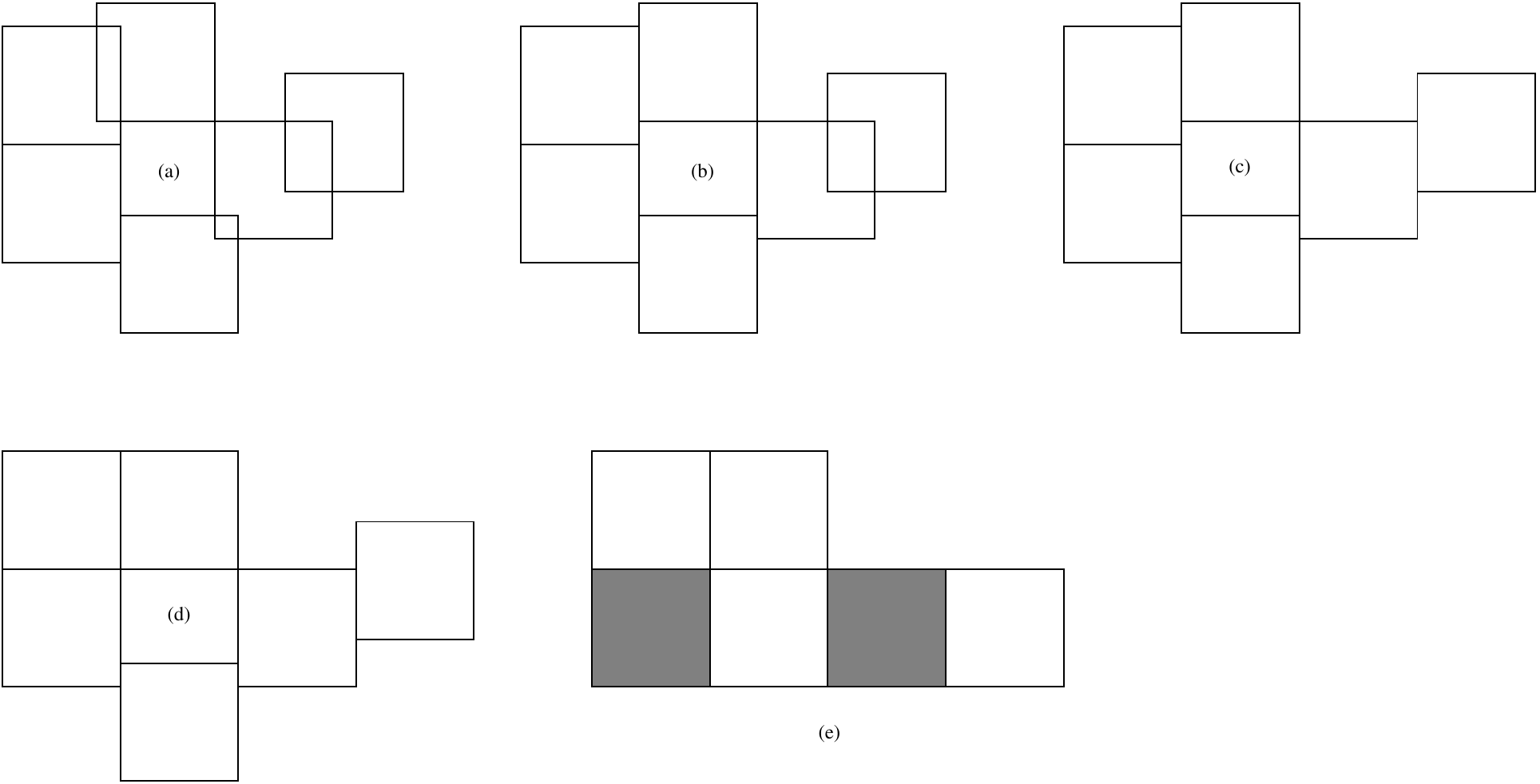}
 \caption{A collection of six squares that are brought into ``grid position''
 by two horizontal and three vertical translations.
 Note that the independence number of the collection strictly decreases here.}
 \label{fig:zalgaller}
\end{figure}

The key observation is that the area covered by $\mathcal{C}(h)$ is an affine (thus also monotone)
function of $h$, at least for $h$ in the interval $[h_-,h_+]$,
where $h_+>0$ is the smallest translation for which a new square becomes
horizontally aligned in $\mathcal{C}(h_+)$, and $h_-<0$ is defined similarly. 
Thus, depending on the monotonicity of this function, either $\mathcal{C}(h_+)$ or
$\mathcal{C}(h_-)$ has larger total area and more horizontally aligned squares than $\mathcal{C}$. It is also
clear that no pair of intersecting squares are dis-joined in the process. On the other hand, new intersecting pairs
or duplicate squares may appear. See Fig.~\ref{fig:zalgaller} for an example.

Now one can repeat the process with a larger body of aligned squares. After finitely many iterations,
we obtain a full collection of horizontally aligned squares. Repeating the procedure vertically
yields a collection in grid position, as intended.
\end{proof}

Note that both arguments yield constructive proofs. In the former, the algorithm finds a collection of parallel slices
that matches the expectation bound by checking a finite set of candidates. In the latter, the algorithm constructs
the integer subgrid induced by the squares in the final displaced position. 

In summary, we have seen three different solutions to the covering problem for \emph{congruent} axis-parallel cubes.
In the next section, we discuss what happens when the cubes are allowed to have different side lengths.

\section{A surprise and further progress}\label{sec:surprise}

Sokolin, Rado, Nordlander, and Zalgaller viewed Theorem~\ref{thm:sharp_covering_congruent_cubes} as a special
case of the conjecture that $F(Q^d)=2^{-d}$, recall~\eqref{eq:F-false}.
E.g., in~\cite{Rado1968III}, the third of R.~Rado's papers on the subject, dated 1968,
this is described as a "long-standing conjecture". Remarkably, in the two-dimensional case this
conjecture has been disproved by M.~Ajtai in 1973~\cite{Ajtai1973}. 

\begin{theorem}[Ajtai's counterexample]
$F(Q^2)<\frac{1}{4}$. 
\end{theorem}
\begin{proof} 
Let $\mathcal{A}$ be a finite collection of axis-parallel squares in the plane,
all contained in an axis-parallel rectangle $R$.
We say that $\mathcal{A}$ is an \emph{almost-counterexample} if the following two properties hold:
\begin{enumerate}
\item Any disjoint sub-collection of $\mathcal{A}$ occupies at most one fourth of the total area of $\mathcal{A}$. 
\item Any disjoint sub-collection of $\mathcal{A}$ containing no square that sits on the lower side of $R$
  occupies strictly less than one fourth of the total area of $\mathcal{A}$.  
\end{enumerate}

We now show that the existence of an almost-counterexample $\mathcal{A}$ implies the existence of a genuine
counterexample $\C$. Let $Q_0=[-1,1]^2$. The family $\C$ consists of four unit squares obtained by halving
the sides of $Q_0$ and four suitably rescaled and rotated copies of $\A$. See Fig.~\ref{fig:ajtai}.

\begin{figure}[htbp]
 \centering
 \includegraphics[scale=0.8]{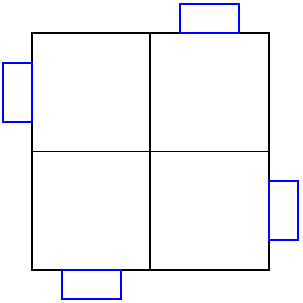}
 \caption{A counterexample obtained from almost-counterexamples $\A$ (drawn schematically as four blue rectangles).}
 \label{fig:ajtai}
\end{figure}

Let us verify that $\mathcal{C}$ is indeed a counterexample. A disjoint sub-collection $\mathcal{S}$ of
$\mathcal{C}$ may contain either one or none of the four unit squares. Let us discuss the two cases
separately. If $\mathcal{S}$ contains none of the unit squares, then property (1) in the definition of
almost-counterexample guarantees that $\mathcal{S}$ has density at most $\frac{1}{4}$ in each of the copies of
$\mathcal{A}$. Thus, it must have density $<\frac{1}{4}$ in the original collection. If $\mathcal{S}$ contains
one of the four unit squares, then $\mathcal{S}$ has density $\frac{1}{4}$ in the four unit squares, at most
$\frac{1}{4}$ in three copies of $\mathcal{A}$ (by property (1)), and strictly less than $\frac{1}{4}$ in the
copy of $\mathcal{A}$ sitting on a side of the selected unit square (by property (2)).  

\begin{figure}[htbp]
 \centering
 \includegraphics[scale=0.6]{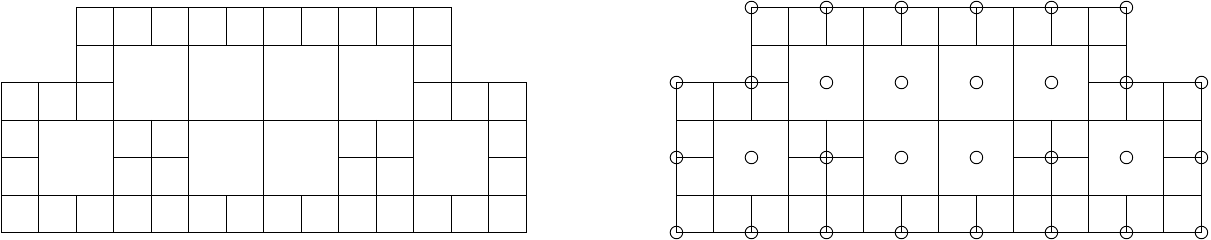}
 \caption{Left: an almost-counterexample with $52$ squares of side $1$ and $2$.
   Right: the same almost-counterexample with the vertices of $19$ "fictitious" squares of side~$2$
   superimposed (only relevant in the analysis). 
} 
 \label{fig:52}
\end{figure}

Thus, we are left with the proof that an almost-counterexample $\mathcal{A}$ exists. 
One is shown in Fig.~\ref{fig:52} (left). Let $\D$ be a disjoint subcollection of  $\mathcal{A}$.
It is clear that for each of the $19$ fictitious $2\times 2$ squares in Fig.~\ref{fig:52} (right),
at most one quarter of its area is covered by $\D$. Thus, property~(1) in the definition of
an almost-counterexample is satisfied.  

In order to show property~(2), it is enough to prove that, if $\D$ contains none of the
small bottom squares, then it covers $0 \%$ of the area of one of the fictitious squares.
We use red to color squares in $\D$ and we mark by 'X' a square if it is \emph{not} in $\D$. 
Refer to rows and columns in Fig.~\ref{fig:52-marked} for the argument; assume the rows are
numbered $1,\ldots,6$ from low to high and the columns are numbered $1,\ldots,14$ from left to right.
\begin{figure}[htbp]
	\centering
	\includegraphics[scale=0.6]{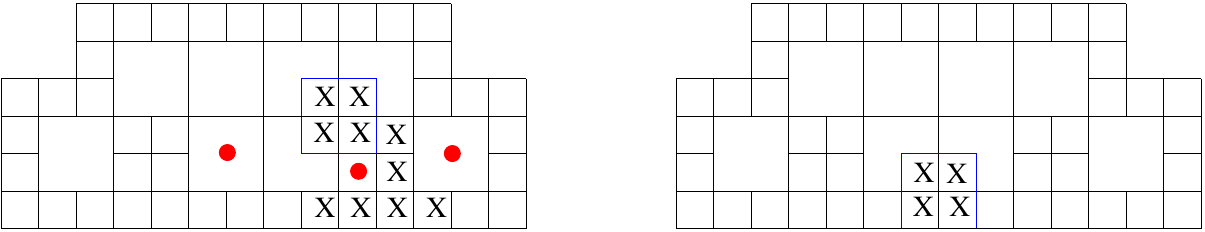}
	\caption{Case~1 (left) and Case~2 (right) of the analysis.
          All the unit squares in the bottom row are considered marked.}
	\label{fig:52-marked}
\end{figure}
We distinguish two cases. 

\smallskip
\emph{Case 1.} The $2\times 2$ square in rows $2,3$ and columns $6,7$ is red.
We will reach a contradiction in a few forced steps: (i)~the unit square in row $2$ and column $10$ is red.
(ii)~the $2\times 2$ square in rows $2,3$ and columns $12,13$ is red. (iii)~Now $0\%$ of the $2\times 2$
fictitious blue square in rows $3,4$ and columns $9,10$ is covered by $\D$.
By symmetry (and the assumptions), a similar analysis works if the $2\times 2$ square in rows $2,3$ and
columns $8,9$ is red. 
 
\smallskip
\emph{Case 2.} Neither the $2\times 2$ square in rows $2,3$ and columns $6,7$, nor the $2\times 2$ square in
rows $2,3$ and columns $8,9$, is red. We reach a contradiction immediately: $0\%$ of the $2\times 2$
fictitious blue square in rows $1,2$ and columns $7,8$ is covered by $\D$.  
\end{proof}

A further refinement of Ajtai's counterexample was obtained in~\cite{Bereg_Dumitrescu_Jiang2010b}.
Specifically, the afore-mentioned system $\A$ has been reduced and a periodic tiling of the plane
with counterexamples made up of squares of size $1$ and $2$ and holes yields that\footnote{The reader
  may observe that with only two types of squares, a fraction of $1/8$ can be easily covered by disjoint squares;
  and according to~\cite[p.~114]{Croft_Falconer_Guy1991}, L.~Mirsky remarked that $8/63$ can be also achieved.}
$f(Q^2) \leq \frac{1}{4} - \frac{1}{384}$.
Interestingly enough, although a positive answer is expected in any higher dimension,
the question whether $F(Q^d)<2^{-d}$ for all $d \geq 2$ has not been addressed in the literature.  

\smallskip
After Ajtai's result, no reasonable conjecture seems to have emerged about the value of $F(Q^2)$, or of
$F(Q^d)$ for general $d$. The question of determining $F(Q^d)$ kept circulating in the discrete geometry community
(e.g., see Problem~D6 in~\cite{Croft_Falconer_Guy1991}). Some progress was already made by R.~Rado,
who proved that $F(Q^d)>(3^d-7^{-d})^{-1}$ (\cite[Theorem 9]{Rado1949I}, which improves~\eqref{eq:minuscule_rado},
valid for any symmetric convex body). Later S.~Bereg, the second named author, and M.~Jiang
improved it to $F(Q^d)\geq (3^d-\frac{1}{2}+o(1))^{-1}$ (\cite[Theorem 1]{Bereg_Dumitrescu_Jiang2010b}).
While both results are based on non-trivial refinements of Vitali's greedy algorithm, they fall short
of achieving an exponential improvement of the gap between Vitali's $3^{-d}$ bound and the $2^{-d}$ upper bound.  

Recently, this gap was closed, modulo a $d\log d$ factor, by the first author~\cite{Dall'Ara25}.
This is a consequence of the following general result.

\begin{theorem}[$F(K)$ vs.~$f(K)$] \label{thm:F-vs-f}
  For any symmetric convex body of dimension $d$, we have
\begin{equation}\label{eq:F_vs_f}
	F(K)\geq \left( e^{-1}+o(1)\right)  \frac{f(K)}{d\log d}.
\end{equation}
\end{theorem}

Since $F(K)\leq f(K)$, the theorem says that the behavior of the two constants $F(K)$ and $f(K)$ is the same, modulo
a small polynomial factor. Combining~\eqref{eq:F_vs_f} with Theorem~\ref{thm:sharp_covering_congruent_cubes},
we obtain:

\begin{corollary}[Almost sharp Vitali covering lemma for axis-parallel cubes]
  \[F(Q^d) \geq \left(e^{-1}+o(1) \right) \frac{2^{-d}}{d\log d}.
  \footnote{The same is true for the Voronoi cell of any lattice. See p.~\pageref{page:voronoi}.}\]
\end{corollary}

Let us give a sketch of the proof of~\eqref{eq:F_vs_f}, referring to~\cite{Dall'Ara25} for more details. 

\begin{proof}[Sketch of proof of \eqref{eq:F_vs_f}]
  Let $\mathcal{C}=\{x_1+r_1 K, \ldots, x_N+r_NK\}$ be an arbitrary finite collection of
  homothetic copies of $K$, and let
\[
\delta(\mathcal{C})=\max \frac{|\cup_{K'\in \mathcal{S}}K'|}{|\cup_{K'\in \mathcal{C}}K'|}, 
\]
where the maximum is over all disjoint sub-collections $\mathcal{S}$ of $\mathcal{C}$. 

As a first step, we want to reduce ourselves to the situation where each $r_j$ is an integer power of
$1+\frac{1}{d}$. Let $r_j'$ be the minimal integral power of $1+\frac{1}{d}$ larger than $r_j$, and consider
the new collection $\mathcal{C}'=\{x_1+r_1' K, \ldots, x_N+r_N'K\}$. It is not difficult to check that, in
passing from $\mathcal{C}$ to $\mathcal{C}'$, the volume of a maximal disjoint sub-collection increases by a
factor not exceeding $\left(1+\frac{1}{d}\right)^d=e+o(1)$; this justifies our choice of "basic scale"
$1+\frac{1}{d}$. Thus,  
\begin{equation}\label{eq:F_vs_f_1}
\delta(\mathcal{C})\geq \left( e^{-1}+o(1)\right) \delta(\mathcal{C}').
\end{equation}

For the second step, let $J$ be an integer, to be determined shortly. 
We partition $\mathcal{C}'$ into $J$ sub-collections $\mathcal{C}_i'$ ($i=0,\ldots, J-1$)
consisting of elements of the form $x+\left(1+\frac{1}{d}\right)^pK$ with $p$ congruent to $i$ mod $J$.
One of these sub-collections, which we denote $\mathcal{C}''$, must occupy a fraction at least
$\frac{1}{J}$ of the total volume of $\mathcal{C}'$. In particular,
\begin{equation}\label{eq:F_vs_f_2} 
\delta(\mathcal{C}')\geq J^{-1}\delta(\mathcal{C}'').
\end{equation} 

If we choose as $J$ the smallest integer $\geq d\log d$, any two distinct radii
$r_1=\left(1+\frac{1}{d}\right)^{p_1}, r_2=\left(1+\frac{1}{d}\right)^{p_2}$ appearing in $\mathcal{C}''$
satisfy the lacunarity condition ($r_2>r_1$)
 \[
\frac{r_2}{r_1}\geq \left(1+\frac{1}{d}\right)^{d\log d}= \left( 1+o(1) \right) d.
\]

The third and final step consists in taking advantage of this lacunarity. We explain this assuming that only
the two radii $1$ and $\frac{1}{d}$ appear in $\mathcal{C}''$. The general case is similar (see~\cite{Dall'Ara25}).
Thus, \[\mathcal{C}''=\{x_1+K, \ldots, x_N+K\}\cup\{y_1+d^{-1}K, \ldots, y_M+d^{-1}K\}=:\mathcal{C}_1\cup\mathcal{C}_2.\]  
We inflate each set in $\mathcal{C}_1$ by a factor $1+\frac{2}{d}$, \[
\mathcal{C}_1':=\left\{x_1+\left(1+\frac{2}{d}\right)K, \ldots, x_N+\left(1+\frac{2}{d}\right)K\right\}, 
\]
and remove from $\mathcal{C}_2$ every set intersecting some element of $\mathcal{C}_1$. Call the resulting
collection $\mathcal{C}_2'\subseteq \mathcal{C}_2$. Notice that any set removed from $\mathcal{C}_2$ is
contained in $\mathcal{C}_1'$ (e.g., by the triangle inequality for the norm corresponding to $K$).  

In passing from $\mathcal{C}_1$ to $\mathcal{C}_1'$, the volume of each set has been increased by a factor
$(1+2d^{-1})^d=e^2+o(1)$. Thus, by definition of $f(K)$, there exist a disjoint sub-collection $\mathcal{S}_1$
of $\mathcal{C}_1$ occupying a fraction at least $(e^{-2}+o(1))f(K)$ of the volume of
$\mathcal{C}_1'$. Again by definition of $f(K)$, there is also a disjoint sub-collection $\mathcal{S}_2$ of
$\mathcal{C}_2'$ occupying a fraction at least $f(K)$ of its volume. By construction, the collection
$\mathcal{S}_1\cup \mathcal{S}_2$ is disjoint and occupies a fraction at least $(e^{-2}+o(1))f(K)$ of the
total volume of $\mathcal{C}''$. This proves that
\begin{equation}\label{eq:F_vs_f_3} 
\delta(\mathcal{C}'')\geq \left( e^{-2}+o(1) \right)f(K). 
\end{equation} 
Combining~\eqref{eq:F_vs_f_1}, \eqref{eq:F_vs_f_2} and~\eqref{eq:F_vs_f_3}, we obtain
$F(K)\geq (e^{-3}+o(1))\frac{f(K)}{d\log d}$.
A more careful choice of the basic scale $1+\frac{1}{d}$ and the parameter $J$ gives
the stated bound~\eqref{eq:F_vs_f}.  
\end{proof}

\section{The case of balls} \label{sec:balls}

In this section we discuss $F(B^d)$ and $f(B^d)$ in the high-dimensional regime, highlighting a
connection with sphere packing problems. See for instance~\cite{Conway-Sloane} for some background
on this classical problem, and~\cite{KLV04, CZ14, CJMS23, Kl25} for recent developments in the
high-dimensional regime of interest here. 

Recall that we have
\[3^{-d}\leq F(B^d) \leq f(B^d) \leq 2^{-d}, \qquad f(B^d)\leq(e+o(1)) \cdot d\log d \cdot F(B^d). \]
After a preliminary version \cite{DD26a} of this survey appeared, the lower bound was improved 
by a linear factor by Xie and Ge~\cite{XG26}, who obtained $f(B^d) = \Omega( d \cdot 3^{-d}$).

In view of the results for cubes discussed above, the following three questions are natural to consider:
\begin{enumerate} [(1)]
\item Does a bound of the form $ F(B^d) \geq b^{-d}$  with $b<3$ hold?
\item Can one take $b=(2+o_{d \to \infty}(1))$?
\item More precisely, what is the value of the constant
  \[  \liminf_{d\rightarrow \infty} \frac{\log F(B^d)}{d} = \liminf_{d\rightarrow \infty} \frac{\log f(B^d)}{d}? \]
 \end{enumerate}
It turns out that (1) has a positive answer: see Subsection~\ref{subsec:lower}, 
while (2) has a negative answer: see Equation~\eqref{eq:upper} below.
Answering~(3) remains elusive for the time being.

\subsection{Upper bound}  \label{subsec:upper}

The problem of estimating from below $f(B^d)$ is a sort of "relative" sphere packing
problem, where one is constrained to choose the balls from a prescribed and completely arbitrary
collection. A basic question then is whether one can achieve the almost-best possible density
$2^{-d-o(d)}$ inside the union of the given balls. This is not the case: the celebrated 
Kabatiansky--Levenshtein upper bound for spherical codes (see~\cite{KL89}) yields an exponential
strengthening of the upper bound $f(B^d)\leq 2^{-d}$, highlighting in particular that the behavior of the
constants $f(B^d)$ and $F(B^d)$ is quite different from their analogues for cubes (perhaps
  not surprisingly, in view of the connection with sphere packing). Specifically, the following upper bound is proved in~\cite{DD26a}:
\begin{equation} \label{eq:upper}
  f(B^d) = O(2.447^{-d}).
\end{equation}

Here we give a more elementary proof that the basic upper bound $2^{-d}$ can be exponentially improved. 
The proof is based on the following result obtained independently by Rankin~\cite{Ran55} and Davenport--Hajos~\cite{DH51}.
The formulation below is from~\cite[Theorem 1]{Ku07}, where it is presented as a special case of a "dispersion problem".

\begin{theorem}[Rankin, Davenport--Hajos]\label{thm:rankin}
If $d+2$ points lie in the unit ball, then at least two of them have distance smaller than
or equal to $\sqrt{2}$. \footnote{It is worth  noting that the above result is consistent
with Jung's Theorem~\cite{Ju1901}, namely the fact that every set of unit diameter in $\RR^d$
can be covered by a ball of diameter $\sqrt{\frac{2d}{(d+1)}}$ in $\RR^d$; see~\cite[p.~46]{HD64}.
One can see that a ball of diameter $2\sqrt{\frac{2d}{(d+1)}}$ contains $d+1$ points at pairwise
distance $2$, namely the vertices of a regular simplex of side length~$2$.}
\end{theorem}

Using the above fact, one easily gets the following.

\begin{theorem}[Elementary upper bound]\label{thm:elementary}
	\[ f(B^d)\leq (d+1)(1+\sqrt{2})^{-d}. \] 
\end{theorem}
Note that $1+\sqrt{2} < 2.447$, so this result is asymptotically weaker than that in~\eqref{eq:upper}.

\begin{proof} 
Let $\mathcal{C}$ be the collection of all unit balls  in $\RR^d$ with centers in a ball of
radius $\sqrt{2}$, that is, all unit balls contained in a ball of radius $1+\sqrt{2}$. We
claim that any disjoint sub-collection $\mathcal{S}\subseteq \C$ has cardinality at most
$d+1$. From this, the thesis follows immediately. The claim is equivalent to (a rescaled
version of) Theorem~\ref{thm:rankin}, because the centers of the balls in $\mathcal{S}$ are
points in a ball of radius $\sqrt{2}$, whose pairwise distances are strictly larger than~$2$.  
\end{proof}

\subsection{Lower bound}  \label{subsec:lower}

The first exponential improvement of the lower bound was obtained in~\cite{DD26b}
by extending the analysis of families with low independence number in~\cite[Section~7]{DD26a}.
Specifically, it has been established that $f(B^d) = \Omega(2.910^{-d})$.
The method in~\cite{DD26b} is constructive and yields a polynomial time algorithm
(in the number of balls and the dimension) for finding a disjoint sub-collection realizing the estimate. Moreover,
the same technique gives similar exponentially improved lower bounds for all symmetric convex bodies satisfying
a uniform convexity assumption, e.g., $\ell^p$-balls for all $p\in (1,\infty)$.

\section{Concluding remarks} \label{sec:remarks}

\subsection{Algorithmic aspects}

It is worth noting that while computing a maximum independent set in a disk graph is $\NP$-complete,
a $(1+\eps)$-approximation can be computed in polynomial time for a fixed $\eps$.
That is, there is a \emph{polynomial time approximation scheme}  (PTAS)
for \emph{maximum weight independent set}  ({\textsc MWIS})
in disk graphs, provided that a disk representation of the graph is given.
The running-time for achieving this approximation  for a disk graph with $n$ disks is $n^{O(\eps^{1-d})}$~\cite{Ch03}
or $n^{O(\eps^{2-2d})}$~\cite{EJS05}. Interestingly enough, these algorithms proceed by finding suitable
hierarchical subdivisions of the space; for instance, the algorithm in~\cite{EJS05} partitions the given
balls into a logarithmic number of levels according to their diameters and uses a recursive
subdivision of the space that allows one to apply a shifting strategy on all levels simultaneously.
By some analogy, the proof of the almost sharp Vitali covering lemma for axis-parallel cubes  (\ie,
Theorem~\ref{thm:F-vs-f}) likewise also partitions the given balls into a small number of levels.

\subsection{Open problems}

We conclude with a few open problems regarding the main subject:

\begin{enumerate}  \itemsep 4pt

\item  Recall that $F(Q^2) \leq \frac{1}{4} - \frac{1}{384}$, as established in~\cite{Bereg_Dumitrescu_Jiang2010b}.
  Can this be improved? This question is part of problem~D6 of~\cite{Croft_Falconer_Guy1991}. 

\item Recall that $f(B^2) \geq \frac{\pi}{8 \sqrt3}$, as established by R. Rado
  (see Theorem 10 of~\cite[Theorem 10]{Rado1949I}). Can this be improved? 

\item It is known that $F(Q^2) \geq 1/\lambda_\textup{square}$, where $\lambda_\textup{square} = 8.4797\ldots$;
  and that $F(B^2) \geq 1/\lambda_\textup{disk}$, where $\lambda_\textup{disk} = 8.3539 \ldots$.
  Both estimates are from~\cite{Bereg_Dumitrescu_Jiang2010b}. Can these bounds be improved? 

\item Prove or disprove that $F(B^2) =1/4$, as conjectured in~\cite{Bereg_Dumitrescu_Jiang2010a}.

\item Prove or disprove that $f(B^2) = F(B^2)$. Prove or disprove that $f(B^d) = F(B^d)$ for every $d \geq 2$. 
Note that $f(B^1) = F(B^1)=1/2$. 

\item Prove or disprove that $F(Q^d) < 2^{-d}$ for every $d \geq 3$. This is also part of Problem~D6
  in~\cite{Croft_Falconer_Guy1991}. 

\item \label{key} By~\cite[Thm.~1.1]{DD26b},  $f(B^d) = \Omega(2.910^{-d})$,
  and by~\cite[Thm.~6.1]{DD26a},  $f(B^d) = O(2.447^{-d})$.
  Can the gap between these bounds be reduced? 

\end{enumerate}

We regard $(\ref{key})$ as our main open problem.

\end{document}